\documentclass{amsproc}

\newtheorem{theorem}{Theorem}[section]

\newtheorem{cor}[theorem]{Corollary}

\theoremstyle{definition}

\theoremstyle{remark}

\numberwithin{equation}{section}

\usepackage{amssymb}
\usepackage{amsmath}
\usepackage{amsfonts}
\usepackage{mathrsfs}
\usepackage{url}

\newcommand{\R}{\mathbf{R}}                             
\newcommand{\N}{\mathbf{N}}                             
                             
\newcommand{\C}{\mathbf{C}}

\newcommand{\calI}{\mathcal{I}}

\newcommand{\frakI}{\mathfrak{I}}

\newcommand{\rmd}{\mathrm{d}}

\newcommand{\scrC}{\mathscr{C}}

\newcommand{\scrL}{\mathscr{L}}

\newcommand{\bfone}{\mathbf{1}}

\newcommand{\dist}{\operatorname{dist}}

\renewcommand{\Re}{\operatorname{Re}}

\newcommand{\ad}{\operatorname{ad}}

\newcommand{\Tr}{\operatorname{Tr}}

\newcommand{\Vol}{\operatorname{Vol}}

\newcommand{\ov}[1]{\overline{#1}}

\newcommand{\bsm}{\left(\begin{smallmatrix}}
\newcommand{\esm}{\end{smallmatrix}\right)}
\newcommand{\bpm}{\begin{pmatrix}}
\newcommand{\epm}{\end{pmatrix}}
\newcommand{\angles}[1]{\left\langle #1\right\rangle}

\newcommand{\pars}[1]{\left( #1\right)}

\newcommand{\dmu}{\:\mathrm{d}\mu}

\begin{document}

\title{Spectral Asymptotics for Operators of H\"{o}rmander Type}

\author{Andrew L. Ursitti}
\address{Department of Mathematics, Purdue University, West Lafayette, Indiana 47907}
\curraddr{Department of Mathematics,
Purdue University, West Lafayette, Indiana 47906}
\email{aursitti@math.purdue.edu}


\date{October 31, 2013.}



\begin{abstract}
An asymptotic equality of the form $\operatorname{Tr}_{L^2} e^{-t(L+V)}=Ct^{-\alpha}+o(t^{-\alpha})$ as $t\rightarrow 0$ is given for the trace of the heat semigroup generated by operators on compact manifolds of the form $L+V=-\sum_{i=1}^{m}X_i^2 +\sum_{i,j=1}^mc_{ij}[X_i,X_j]+\sum_{i=1}^m \gamma_iX_i+V$ for smooth real potentials $(V)$ which satisfy H\"{o}rmander's bracket-generating condition. In the self-adjoint case, a Weyl law is proved for the spectra of such operators. Analogous results are proved for the Dirichlet boundary value problem.
\end{abstract}

\maketitle

\section{Introduction}

In this article we shall be concerned with a general class of second-order scalar differential operators on a compact manifold $M$ of dimension $n\geq 3$.  More precisely, we will consider operators $L$ of the form 
\begin{equation}\label{asdfjkl;100}
L=-\sum_{i=1}^{m}X_i^2 +\sum_{i,j=1}^mc_{ij}[X_i,X_j]+\sum_{i=1}^m \gamma_iX_i
\end{equation}
on $M$, where $X_i$ is a smooth real vector field on $M$ and $c_{ij},\gamma_i\in\scrC^\infty(M;\R)$ for every $i,j\in\{1,\ldots, m\}$. More generally, we will also consider operators of the form $L+V$ where $L$ is as described above and $V\in\scrC^\infty(M;\R)$. The presence of the arbitrary smooth potential $V$ renders unnecessary the separate consideration of those operators which can be written in the form (\ref{asdfjkl;100}) only locally, because any such operator can be written \emph{globally} as $L+V$ where $L$ and $V$ are as described above. Here and below, $\mu$ will denote a generic volume density for $M$ which is always assumed to be smooth and nondegenerate but which is subject to no additional assumptions. All notions of adjoint will refer to the Hermitian scalar product in $L^2(M,\mu)$. 

Concerning such operators we will assume that the H\"{o}rmander condition holds: for all $x\in M$, any tangent vector at $x\in M$ is the restriction to $T_xM$ of an element of the Lie algebra generated over $\R$ by $\{X_1,\ldots, X_{m}\}$. For a given word $I=(i_1,\ldots, i_j)\in \{1,\ldots, m\}^j$ we define $|I|=j$. Additionally, we define the vector field $X_{I}$ and the subspace $T_{x,j}M\subset T_xM$ for any point $x\in M$ by
\[
X_{I}=\ad(X_{i_1})\cdots \ad(X_{i_{j-1}})X_{i_j}\quad \mbox{and}\quad T_{x,j}M =\angles{\{X_{I}|_x\in T_xM:|I|\leq j\}},
\]
where $\angles{\cdot}$ indicates linear algebraic closure. 

The function $d:M\times \N\rightarrow \N$ given by $d(x,j) =\dim T_{x,j}M$ is nondecreasing in $j$ for fixed $x$ and lower semi-continuous in $x$ for fixed $j$. Additionally, the H\"{o}rmander condition can be restated as follows: there exists an integer function $\tau:M\rightarrow\N$ such that $d(x, \tau(x)) = n$; it will be assumed that $\tau$ is minimal among all functions satisfying this property. A limit point argument using compactness and lower semi-continuity shows that $\tau$ is bounded. In other words there is a \emph{global} maximum degree $\tau_L=\max \tau(\cdot)$ of iterated Lie brackets of the $X_i$ necessary to generate any given tangent space of $M$.

The \emph{homogeneous dimension} $Q:M\rightarrow \N$ is defined by 
\[
Q(x)=d(x,1)+\sum_{j=2}^{\infty}j[d(x,j)-d(x,j-1)]=\tau_L n-\sum_{j=1}^{\tau_L-1}d(x,j).
\]

The expression on the right shows that the function $Q$ is upper semi-continuous and bounded. Furthermore, within any given level set of $Q$ the level sets of $d(\cdot, j)$ must be relatively open and therefore topologically separated. In other words, the integer vector $(d(\cdot,1),\ldots, d(\cdot, \tau_L-1))$ is constant on the connected components of the level sets of $Q$. The significance of $Q(x)$ can be summarized briefly as follows: in addition to the standard pointwise tangent space to the manifold $M$, the nature of the operator $L$ implies the existence of a tangential nilpotent Lie group at every point $x\in M$, and $Q(x)$ is the Hausdorff dimension of the natural metric structure on this group - a notion which is ubiquitous in the literature on operators of this type (\cite{MR0436223}).

H\"{o}rmander's condition along with the smoothness of $V$ imply that $L+V$ is hypoelliptic and it therefore possesses a smooth, positive heat kernel $E^{L,\mu}_V(\cdot;\cdot,\cdot):M\times M\times (0,\infty)\rightarrow (0,\infty)$. In general, $E^{L,\mu}_V(\cdot;\cdot,\cdot)$ will not be symmetric in the physical variables so we must specify their roles: it will be assumed that for every $x\in M$, $E^{L,\mu}_V(x;\cdot,\cdot)$ solves $(L+V+\partial_t)E^{L,\mu}_V(x;\cdot,\cdot)\equiv 0$ on $M\times (0,\infty)$ and converges to the Dirac mass $\delta_x$ as $t\rightarrow 0$ with respect to the density $\mu$.

Concerning the small time asymptotics of $E^{L,\mu}_V$, the following result has been established using probabilistic techniques by Takanobu (\cite{MR944857}), and Ben Arous and L{\'e}andre (\cite{MR1128069} and \cite{MR1133372}) in the case $V\equiv 0$ and also independently by Ben Arous (\cite{MR1011978}) in the case where $-L$ is a pure sum of squares (i.e. without drift or potential).
\begin{theorem}[\cite{MR944857},\cite{MR1128069},\cite{MR1133372}]\label{asdfjkl;101} There exist sequences of measurable functions $c^{L,\mu}_i:M\rightarrow \R$ and $r^{L,\mu}_i:M\times(0,\infty)\rightarrow \R$, $i\geq 0$, such that for $N\geq 0$
\begin{equation}\label{asdfjkl;102}
t^{Q(x)/2}E^{L,\mu}_0(x;x,t)=\sum_{i=0}^Nc^{L,\mu}_i(x)t^i+r^{L,\mu}_N(x,t)
\end{equation}
and $r^{L,\mu}_N(x,t)=o(t^N)$ pointwise at every $x\in M$ as $t\rightarrow 0$. In addition, $c^{L,\mu}_0$ is strictly positive.
\end{theorem}

It should be mentioned at this point that in \cite{MR1128069}, \cite{MR1133372}, and \cite{MR944857} this result is proved for generic operators of the form $-\sum_{i=1}^m X_i^2+Y$ without any restrictions on $Y$ other than the requirement that the entire set \newpage $\{Y,X_1,\ldots, X_m\}$ must satisfy the H\"{o}rmander condition.\footnote{In this case one must define $Q$ using an alternative definition of $T_{x,j}M$ which gives weight two to all or part of the drift vector $Y$ (for instance, in the case of the full heat operator $\scrL=L+V+\partial_t$, it would be appropriate to treat $\partial_t$ as an operator of order two). This amounts to a modification of the standard power-order filtration of differential operators and is the correct way to account for the anisotropy produced by the mixed orders of differentiation when $Y$ cannot be written using iterated commutators of the $X_i$.} However, the positivity of $c^{L,\mu}_0$ relies on the special form of $Y$ which we have taken as a hypothesis. In fact, the heat kernel is quite pathological if $Y$ is not of this form since in that case the drift will overpower the diffusion and cause $E^{L,\mu}_0(x;x,t)$ to \emph{vanish} as $t\rightarrow 0$ faster than any power of $t$. See \cite{MR1128069},\cite{MR1133372} and \cite{MR962859} for details on this point.

Let $Q_L=\max  Q(\cdot)$. It will be useful to reinterpret Theorem \ref{asdfjkl;101} by using (\ref{asdfjkl;102}) to express $t^{Q_L/2}E^{L,\mu}_0(x;x,t)$ rather than $t^{Q(x)/2}E^{L,\mu}_0(x;x,t)$. This is done as follows: if $i=0,1,2,3,\ldots$ then let $c_i^{L,\mu}$ be given by Theorem \ref{asdfjkl;101}, if $i$ is a half-integer or a negative integer then set $c_i^{L,\mu}\equiv 0$ and define for $j\geq 0$, $\epsilon_{j/2}^{L,\mu}=\sum_{k=n}^{Q_L} c_{(j+k-Q_L)/2}^{L,\mu}\bfone_{\{Q=k\}}$. With these definitions, $\epsilon_{j/2}^{L,\mu}\equiv 0$ if $j<Q_L-Q(x)$, $\epsilon_{(Q_L-Q(x))/2}^{L,\mu}(x)=c_0^{L,\mu}(x)>0$ and (\ref{asdfjkl;102}) implies
\begin{equation}\label{asdfjkl;124}
t^{Q_L/2}E^{L,\mu}_0(x;x,t)=\sum_{j=0}^N \epsilon_{j/2}^{L,\mu}(x)t^{j/2}+R_{N/2}^{L,\mu}(x,t)
\end{equation}
with $R_{N/2}^{L,\mu}(x,t)=o(t^{ N/2})$ pointwise as $t\rightarrow 0$.

Denote $F_k=\{x:Q(x)\geq k\}$ which is a closed set due to the upper semi-continuity of $Q$. As a consequence of (\ref{asdfjkl;124}) we have
\begin{equation}\label{asdfjkl;103}
\lim_{t\rightarrow 0} t^{Q_L/2}E^{L,\mu}_0(x;x,t)=\epsilon^{L,\mu}_0(x)=c^{L,\mu}_0(x)\bfone_{F_{Q_L}}(x)
\end{equation}
for every $x\in M$. In general this limit will not be achieved uniformly. However, 
in section 2 it will be shown (Corollary \ref{asdfjkl;115}) that $t^{Q_L/2}E^{L,\mu}_0(x;x,t)$ is bounded above by an absolute constant $k_L$, uniformly in $(x,t)\in M\times (0,1]$. On combining this estimate with (\ref{asdfjkl;103}) we have
\begin{equation}\label{asdfjkl;106}
\lim_{t\rightarrow 0}\int t^{Q_L/2}E^{L,\mu}_0(x;x,t) \dmu(x)=\int \epsilon^{L,\mu}_0\dmu
\end{equation}
by dominated convergence.

In section 3 it will be shown that (\ref{asdfjkl;106}) holds, with the same righthand side, with $E^{L,\mu}_V$ in place of $E^{L,\mu}_0$ for any nonzero $V\in\scrC^\infty(M;\R)$. The resulting limit can be recast as a theorem on the asymptotic behavior of the heat trace for $L+V$: 

\begin{theorem} \label{asdfjkl;114} The asymptotic equality 
\begin{equation} 
\Tr_{L^2}e^{-t(L+V)}=\pars{\int \epsilon^{L,\mu}_0\dmu}t^{-Q_L/2}+o(t^{-Q_L/2})
\end{equation}
holds as $t\rightarrow 0$.
\end{theorem}

It is possible to have $\int \epsilon^{L,\mu}_0\dmu=0$, in which case the theorem gives only the rather imprecise result that $\Tr_{L^2}e^{-t(L+V)}=o(t^{-Q_L/2})$. However, according to Theorem \ref{asdfjkl;101} since $\epsilon^{L,\mu}_0=c^{L,\mu}_0\bfone_{F_{Q_L}}$, $\int \epsilon^{L,\mu}_0\dmu\neq 0$ if and only if $F_{Q_L}$ has positive measure, in which case $\int \epsilon^{L,\mu}_0\dmu> 0$. 

There is a certain sense in which Theorem \ref{asdfjkl;114} is somewhat less trivial than this short description of its proof suggests. For instance, the coefficients $\epsilon^{L,\mu}_{j/2}$ and remainders $R^{L,\mu}_{N/2}$ are continuous on level sets of $Q$ and smooth on their interiors, and the pointwise asymptotic equality $R^{L,\mu}_{N/2}(x,t)=o(t^{N/2})$ is uniform on compact subsets of level sets of $Q$ (see \cite{MR1011978} for proofs of these facts in the zero drift case). Thus, if $Q(\cdot)\equiv Q_L$ is constant then the half-integer terms are irrelevant and the remainder estimates are globally uniform, and therefore setting $i=j/2$ for even $j$ gives
\begin{equation}\label{asdfjkl;150}
\Tr_{L^2}e^{-tL}=\sum_{i=0}^N\pars{\int \epsilon^{L,\mu}_i\dmu}t^{i-Q_L/2}+o(t^{N-Q_L/2})
\end{equation}
for every $N\geq 0$, just as in the elliptic case ($Q(\cdot)\equiv Q_L=n$). However, if $Q$ is not constant then there seems to be no \textit{a priori} reason to suspect that the integrated remainders at order $N/2$ obey $\int R^{L,\mu}_{N/2}(\cdot, t)\dmu=o(t^{N/2})$. Therefore, one cannot na\"{i}vely expect (\ref{asdfjkl;150}) to be true if $Q$ is not constant since pathological effects near the boundaries of the level sets of $Q$ could conceivably contribute to the heat trace in the $t\rightarrow 0$ limit. 

In fact, this phenomenon does occur in the higher order terms of the heat trace for elliptic boundary value problems (\cite{MR2040963},\cite{MR0217739}), so in the sense that the function $Q$ will change discontinuously when nonconstant, there is all the more reason to suspect that perhaps similar effects might occur on the boundaries of the level sets of $Q$ - even though the problem is posed on a manifold without boundary. Therefore, one might explain Theorem \ref{asdfjkl;114} by saying that if $F_{Q_L}$ has positive measure then any such pathological effects cannot affect the term of lowest order in the heat trace, as is the case even with boundary value problems for elliptic operators. A heuristic explanation as to why this should be true is provided after the proof of Theorem \ref{asdfjkl;111}.

Concerning spectral asymptotics, our main result is the following Weyl law for the spectrum of $L+V$:

\begin{theorem}\label{asdfjkl;107} If $L$ is formally self-adjoint and if $N(\lambda,L+V)$ denotes the number of eigenvalues of $L+V$ which are not greater than $\lambda$, then the asymptotic equality
\begin{equation}\label{asdfjkl;139}
N(\lambda,L+V)= \frac{\int \epsilon^{L,\mu}_0\dmu}{\Gamma(Q_{L}/2+1)}\lambda^{Q_L/2}+o(\lambda^{Q_L/2})
\end{equation}
holds as $\lambda\rightarrow \infty$.\end{theorem}

Again, a necessary and sufficient condition for the positivity of the spectral coefficient $\int \epsilon^{L,\mu}_0\dmu$ is that $F_{Q_L}$ has positive measure. In the Riemannian case $Q_\Delta=n$, and computations involving elementary properties of the Riemannian metric and volume form demonstrate that $\epsilon_0^\Delta$ takes the constant value $(4\pi)^{-n/2}$, so one recovers the familiar Riemannian Weyl law in which the volume of the manifold determines the spectral coefficient. It seems that the correct interpretation of this idea in the present context is that $\int \epsilon^{L,\mu}_0 \dmu=\int c^{L,\mu}_0\bfone_{F_{Q_L}}\dmu$ is the volume of $F_{Q_L}$ in the volume density $h\mu$, where $h$ is any positive function which coincides with $c^{L,\mu}_0$ on $F_{Q_L}$. 
With this in mind, (\ref{asdfjkl;139}) can be rewritten as
\begin{equation}\label{asdfjkl;140}
N(\lambda,L+V)= \frac{\Vol_{h\mu}(F_{Q_L})}{\Gamma(Q_{L}/2+1)}\lambda^{Q_L/2}+o(\lambda^{Q_L/2})
\end{equation}
where $h$ coincides with $c^{L,\mu}_{0}$ on $F_{Q_L}$ and $\Vol_{h\mu}(\cdot)$ computes the volume of a Borel set in the density $h\mu$.

The spectral coefficient $\Vol_{h\mu}(F_{Q_L})/\Gamma(Q_L/2+1)$ exhibits interesting sensitivity to certain perturbations of the operator $L$. For instance, assuming that $L$ and $L'$ are both of the form (\ref{asdfjkl;100}) and that $L$ is self-adjoint, a short computation demonstrates that for any $\psi\in\scrC^\infty(M;\R)$, $T_\psi=L+\psi^2L'/2+(\psi^2L')^\ast/2$ is also self-adjoint and of the form (\ref{asdfjkl;100}) within a perturbation by a smooth real potential. In particular, (\ref{asdfjkl;140}) applies to $T_\psi+V$ for any $V\in\scrC^\infty(M;\R)$. Let us assume that $Q_{L}>Q_{L'}$ (this is the case if, for instance, $L'$ is elliptic and $L$ is not), whence
\[
N(\lambda,T_\psi+V) = \frac{\Vol_{h\mu}(F_{Q_L}\cap \{\psi=0\})}{\Gamma(Q_{
L}/2+1)}\lambda^{Q_L/2}+o(\lambda^{Q_L/2})
\]
where, as before, $h$ is any positive function which coincides with $c^{L,\mu}_{0}$ on $F_{Q_L}$. Thus, one can ``control" the spectral coefficient through suitable alteration of the support of $\psi$.

Before proving these results, let us comment on the existing literature related to this problem. Asymptotic laws for the spectra of \emph{elliptic} operators are known classically, beginning with the well-known work of H. Weyl for the Dirichlet problem in planar domains. The state of the art for elliptic problems is available in \cite{MR2040963}, among other places. Concerning the non-elliptic case, Menikoff and Sj\"{o}strand (\cite{MR547016},\cite{MR0481627},\cite{MR564905},\cite{MR555302},\cite{MR520875}) have obtained spectral asymptotics for a large class of hypoelliptic pseudodifferential operators not restricted to order two. However, they deal only with operators which are subelliptic with a loss of at most one derivative and our techniques apply to operators which can lose any number of derivatives strictly less than two (i.e. any number of derivatives strictly less than the order of the operator in question) in the Sobolev scale. 

The work of Fefferman and Phong (\cite{MR589278},\cite{MR730094}) and M{\'e}tivier (\cite{MR0427858}) is much more closely related to the present work. In \cite{MR730094}, Fefferman and Phong consider generic self-adjoint second order operators: $L=-a_{ij}(x)\partial_{x_j}\partial_{x_i}+b_i(x)\partial_{x_i}+c(x)$ locally (i.e. not simply those of the form $\sum X_i^\ast X_i+c$) with smooth coefficients which satisfy a subelliptic estimate of the form $\angles{Lu,u}+K\|u\|^2_{L^2}\geq \|u\|^2_{H^\epsilon}$. For  operators of this type they obtain the existence of positive constants $C_1\leq C_2$ such that 
\[
C_1\int \mu(B_L(x,\lambda^{-1/2}))^{-1}\dmu(x)\leq N(\lambda,L)\leq C_2\int \mu(B_L(x,\lambda^{-1/2}))^{-1}\dmu(x)
\]
for large $\lambda$. The notation $B_L$ indicates the subunit metric balls which are naturally associated to $L$ and which will be described in detail in section 2. Thus, whereas our Theorem \ref{asdfjkl;107} gives an exact spectral asymptotic law rather than simply an estimate, it applies to a much smaller class of operators than that which is considered in \cite{MR730094}. 

Finally, in \cite{MR0427858} M{\'e}tivier has obtained Theorem \ref{asdfjkl;107} for the Friedrichs extensions on precompact domains in general smooth manifolds (and thus on the entirety of a given manifold, provided that it is compact) of operators of the form $\sum_{i=1}^m X_i^\ast X_i+V$ under the additional assumption that the homogeneous dimension $Q$ is constant. Theorem \ref{asdfjkl;107} can therefore be seen as a more or less direct generalization of M{\'e}tivier's results in the case of compact manifolds. For the Dirichlet boundary value problem, an analogous generalization is presented in section 4. The author would like to thank his advisor, Fabrice Baudoin, for suggesting this topic of research.

\section{A Uniform Diagonal Estimate for $E^{L,\mu}_0$ in Small Times}
While the infinite differentiability of $t\mapsto t^{Q(x)/2}E^{L,\mu}_0(x;x,t)$ is a noteworthy fact, we will only be concerned with the existence and positivity of the pointwise limit $\epsilon^{L,\mu}_0(x)=\lim_{t\rightarrow 0}t^{Q_L/2}E^{L,\mu}_0(x;x,t)$. In this section we will prove that $t^{Q_L/2}E^{L,\mu}_0(x;x,t)$ is bounded uniformly in $M\times(0,1]$ so that the dominated convergence theorem can be used to compute the $t\rightarrow 0$ limit of the integral 
\[
\int_M t^{Q_L/2}E^{L,\mu}_0(x;x,t) \dmu=t^{Q_L/2}\Tr_{L^2} e^{-tL}.
\]

This estimate will be deduced as a corollary of the following theorem:

\begin{theorem}\label{asdfjkl;111}For any fixed $x_\ast\in M$, there exists an open neighborhood $N_{x_\ast}\ni x_\ast$ and a positive constant $k(x_\ast)>0$ such that 
$E^{L,\mu}_0(x;x,t)\leq k(x_\ast)t^{-Q(x_\ast)/2}$ provided that $x\in N_{x_\ast}$ and $0<t\leq 1$. 
\end{theorem}

In other words, the estimate of the form $E^{L,\mu}_0(x_\ast;x_\ast,t)\leq c(x_\ast)t^{-Q(x_\ast)/2}$ for small $t$ which follows immediately from Theorem \ref{asdfjkl;101}, holds in a full open neighborhood of $x_\ast$ provided that the constant $c(x_\ast)$ is appropriately modified. Assuming for the moment that Theorem \ref{asdfjkl;111} holds, we can state and prove the aforementioned corollary:

\begin{cor} \label{asdfjkl;115} There exists an absolute constant $k_L>0$ such that 
\[
t^{Q_L/2}E^{L,\mu}_0(x;x,t)\leq k_L
\]
for all $x\in M$ and $0<t\leq 1$.\end{cor}

\begin{proof}
From Theorem \ref{asdfjkl;111}, $t^{Q_L/2}E^{L,\mu}_0(x;x,t)\leq k(x_\ast)t^{Q_L/2-Q(x_\ast)/2}$ for all $x\in N_{x_\ast}$ provided $0<t\leq 1$. However, since $Q_L=\max Q(\cdot)$, the exponent $Q_L/2-Q(x_\ast)/2$ is nonnegative so that in fact $t^{Q_L/2}E^{L,\mu}_0(x;x,t)\leq k(x_\ast)$ for $0<t\leq 1$. Since $M$ is assumed to be compact, there must exist a finite set $\{x_1,\ldots , x_q\}\subset M$ such that $M\subset N_{x_1}\cup\ldots \cup N_{x_q}$. Thus, a generic point $x\in M$ is an element of $N_{x_l}$ for some $l\leq q$ and therefore the corollary holds with $k_L=\max \{ k(x_l): 1\leq l\leq q\}$.
\end{proof}

The rest of this section will be devoted to the proof of Theorem \ref{asdfjkl;111}. To begin the proof we recall the well known fact that the heat kernel $E^{L,\mu}_0$ can be controlled by the $\mu$-volumes of certain metric balls. To explain the situation, let us define for any connected open subset $U\subset M$, three sets of curves indexed by $\delta\geq 0$ and $x,y\in U$:
\begin{enumerate}
\item $C^U(x,y,\delta)$: this set consists of all absolutely continuous curves $c:[0,1]\rightarrow U$ such that $c(0)=x$, $c(1)=y$, and such that for almost every $t\in [0,1]$, $c'(t)$ can be written as $\sum_{|I|\leq \tau_L}a_I(t)X_I$ with $|a_I(t)|<\delta^{|I|}$,
\item $C^U_\infty(x,y,\delta)$: this set consists of all absolutely continuous curves $c:[0,1]\rightarrow U$ such that $c(0)=x$, $c(1)=y$, and such that for almost every $t\in [0,1]$, $c'(t)$ can be written as $\sum_{i=1}^m a_i(t)X_i$ with $\max_{1\leq i\leq m} |a_i(t)|<\delta$,
\item $C^U_2(x,y,\delta)$: this set consists of all absolutely continuous curves $c:[0,1]\rightarrow U$ such that $c(0)=x$, $c(1)=y$, and such that for almost every $t\in [0,1]$, $c'(t)$ can be written as $\sum_{i=1}^m a_i(t)X_i$ with $\sum_{i=1}^m |a_i(t)|^2<\delta^2$.
\end{enumerate}

The distance $\rho^U:U\times U\rightarrow [0,\infty]$ is defined by
\[
\rho^U(x,y)=\inf\{\delta>0:C^U(x,y,\delta)\neq \varnothing\}
\]
and $\rho^U_\infty,\rho^U_2:U\times U\rightarrow [0,\infty]$ are defined in an analogous manner.\footnote{This definition of the metric $\rho=\rho^U$ coincides with that which is given by Nagel, Stein and Wainger in \cite{MR793239}, and $\rho^U_\infty$ coincides with $\rho_4$ as defined in that article.} All three distances are in fact finite. This is clear for $\rho^U$ and it follows from the H\"{o}rmander condition and the Chow-Rashevskii theorem (\cite{MR1867362}) for $\rho^U_\infty$ and $\rho^U_2$. The associated metric balls will be denoted $B^U$, $B^U_\infty$ and $B^U_2$ (the balls $B^U_2$ are more or less the same as the subunit balls which were mentioned in the introduction). Moreover, if $x,y\in U$ then
\begin{equation}\begin{array}{rl}
& C^U_\infty(x,y,\delta/\sqrt{m})\subset C^U_2(x,y,\delta)\subset C^U_\infty(x,y,\delta) \subset C^U(x,y,\delta), \\
\mbox{so} &
\rho^U(x,y)\leq \rho^U_\infty(x,y)\leq \rho^U_2(x,y)\leq  \sqrt{m} \rho^U_\infty(x,y), \\
\mbox{and therefore} &
B^U_\infty(x,\delta/\sqrt{m})\subset B^U_2(x,\delta)\subset B^U_\infty(x,\delta)\subset B^U(x,\delta)
\end{array}\label{asdfjkl;600}\end{equation}
 for all $x,y\in U$, $\delta\geq 0$. 

Let the symbol $\calI$ denote a multi-word $\calI=\{I_1,\ldots ,I_n\}$ of length $n$. In other words, $\calI$ denotes a set of $n$ elements, each of which is itself a word $I_k\in\{1,\ldots, m\}^j$ which references an iterated Lie bracket $X_{I_k}$ as defined in the introduction. Let $\frakI$ denote the collection of all multi-words $\calI=\{I_1,\ldots ,I_n\}$ of length $n$ in which $|I_k|\leq \tau_L$ for each constituent word $I_k$ and define for every smooth coordinate system $(W, \{x_i\}_{i\leq n}\subset \scrC^\infty(W;\R))$ on $M$ a function $\Lambda(W,\{x_i\};\cdot,\cdot):W\times[0,\infty)\rightarrow [0,\infty)$ by
\[
\Lambda(W,\{x_i\};x,\delta)=\sum_{\calI\in\frakI}|\det(\calI|_x)|\delta^{\deg (\calI)},
\]
where for every $\calI=\{I_1,\ldots ,I_n\}\in\frakI$, $\deg(\calI)=|I_1| +\cdots +|I_n|$ and $\det(\calI|_x)=\det(X_{I_1}|_x,\ldots, X_{I_n}|_x)$ denotes the determinant of the matrix which brings the ordered coordinate basis $(\partial_{x_1}|_x,\ldots,\partial_{x_n}|_x)$ onto the ordered list $(X_{I_1}|_x,\ldots, X_{I_n}|_x)$. Many of these determinants will be zero, but the H\"{o}rmander condition guarantees that some of them will be nonzero, so $\Lambda(W,\{x_i\};x,\delta)$ is strictly positive and finite for all $x\in W $ and $0<\delta<\infty$.

For the remainder of the proof, let $x_\ast\in M$ denote a fixed yet arbitrary point in $M$. To prove Theorem \ref{asdfjkl;111} we must exhibit a neighborhood $N_{x_\ast}$ of $x_\ast$ and a constant $k(x_\ast)$ such that the conclusions of the theorem are verified. To do this, let $W^0_{x_\ast}$ and $W^1_{x_\ast}$ be connected open subsets of $M$ such that $x_\ast\in W^1_{x_\ast}\subset\subset W^0_{x_\ast}$ and such that $W^0_{x_\ast}$ is a coordinate domain with coordinates $\{x_i\}_{i\leq n}$. In \cite{MR865430}, it was proved that there exists $c_1>0$ such that $E^{L,\mu}_0(x;x,t)\leq c_1/\mu(B_2^M(x,\sqrt{t}))$ for all $(x,t)\in M\times (0,\infty)$.\footnote{The notation $B^M_2$ refers to the specific instance associated to $U=M$ of the generically constructed metric balls $B^U_2$.} Using this inequality, Theorem \ref{asdfjkl;111} will be proved by way of a rather lengthy factor-by-factor estimation of the righthand side of the following inequality:
\begin{align}
E^{L,\mu}_0(x;x,t) &\leq \frac{c_1}{\mu(B_2^M(x,\sqrt{t}))} \notag \\
&=c_1
\frac{\mu(B^{W^1_{x_\ast}}_2(x,\sqrt{t}))}{\mu(B^M_2(x,\sqrt{t}))}
\frac{|B^{W^1_{x_\ast}}_2(x,\sqrt{t})|}{\mu(B^{W^1_{x_\ast}}_2(x,\sqrt{t}))}
\frac{|B^{W^1_{x_\ast}}_\infty(x,\sqrt{t}/\sqrt{m})|}{|B^{W^1_{x_\ast}}_2(x,\sqrt{t})|} 
 \label{asdfjkl;601}\\
&\hspace{1cm}\times
\frac{\Lambda({W^1_{x_\ast}},\{x_i\};x,\sqrt{t})}{|B^{W^1_{x_\ast}}_\infty(x,\sqrt{t}/\sqrt{m})|}
\frac{1}{\Lambda({W^1_{x_\ast}},\{x_i\};x,\sqrt{t})}. \label{asdfjkl;602}
\end{align}
This inequality holds for $x\in W^1_{x_\ast}$ and $0<t<\infty$. To see this, observe that the coordinates $\{x_i\}$ restrict to $W^1_{x_\ast}$ and therefore $\Lambda(W^1_{x_\ast},\{x_i\},\cdot,\sqrt{t})=\Lambda(W^0_{x_\ast},\{x_i\},\cdot,\sqrt{t})$ on $W^1_{x_\ast}$. The notation $|\cdot|$ indicates the measure of a Borel set in the coordinate volume form $\rmd x_1\wedge\cdots \wedge \rmd x_n$, so the assumption that $W^1_{x_\ast}$ is compactly contained in $W^0_{x_\ast}$ ensures that all expressions involving $|\cdot|$ are finite and positive for $0<t<\infty$. Likewise, since the density $\mu$ is assumed to be smooth and nondegenerate, there exists a smooth function $h:W^0_{x_\ast}\rightarrow (0,\infty)$ such that $\mu=h\rmd x_1\wedge\cdots \wedge \rmd x_n$, so $h$ must be bounded above and below by positive numbers on $W^1_{x_\ast}$ and therefore all expressions involving $\mu(\cdot)$ are finite and positive for $0<t<\infty$. 

Now that it is clear that the righthand side is a well defined positive number for $0<t<\infty$, the validity of the inequality is easily seen since each expression which does not appear in the initial inequality $E^{L,\mu}_0(x;x,t)\leq c_1/\mu(B_2^M(x,\sqrt{t}))$ from \cite{MR865430} appears on the righthand side exactly once in both the numerator and denominator. We will estimate each factor in the righthand side in sequence, beginning with the first factor in line (\ref{asdfjkl;601}). Since $W^1_{x_\ast}\subset M$, $B^{W^1_{x_\ast}}_2(x,\sqrt{t})\subset B^M_2(x,\sqrt{t})$ so $\mu(B^{W^1_{x_\ast}}_2(x,\sqrt{t}))/\mu( B^M_2(x,\sqrt{t}))\leq 1$. Furthermore, a short computation shows that 
$|B^{W^1_{x_\ast}}_2(x,\sqrt{t})|/\mu(B^{W^1_{x_\ast}}_2(x,\sqrt{t}))\leq 1/\inf_{W^1_{x_\ast}} h$. Finally, from (\ref{asdfjkl;600}) it is clear that $B^{W^1_{x_\ast}}_\infty(x,\sqrt{t}/\sqrt{m})\subset B^{W^1_{x_\ast}}_2(x,\sqrt{t})$ so $|B^{W^1_{x_\ast}}_\infty(x,\sqrt{t}/\sqrt{m})|/|B^{W^1_{x_\ast}}_2(x,\sqrt{t})|\leq 1$; so after estimating each of the 
three factors we have the following estimate for the product in line (\ref{asdfjkl;601}):
\begin{equation}\label{asdfjkl;604}
c_1
\frac{\mu(B^{W^1_{x_\ast}}_2(x,\sqrt{t}))}{\mu(B^M_2(x,\sqrt{t}))}
\frac{|B^{W^1_{x_\ast}}_2(x,\sqrt{t})|}{\mu(B^{W^1_{x_\ast}}_2(x,\sqrt{t}))}
\frac{|B^{W^1_{x_\ast}}_\infty(x,\sqrt{t}/\sqrt{m})|}{|B^{W^1_{x_\ast}}_2(x,\sqrt{t})|} \leq\frac{c_1}{\inf_{W^1_{x_\ast}} h}
\end{equation}
for all $x\in W^1_{x_\ast}$ and all $t>0$.

Proceeding to the first factor in line (\ref{asdfjkl;602}), observe that in \cite{MR793239} Nagel, Stein and Wainger have proved that there exists an open set $W^2_{x_\ast}\subset W^1_{x_\ast}$ which contains $x_\ast$ such that $\rho^{W^1_{x_\ast}}_\infty(\xi,y)\leq C(W^2_{x_\ast})\rho^{W^1_{x_\ast}}(\xi,y)$ for some $C(W^2_{x_\ast})>0$ and all $\xi,y\in W^2_{x_\ast}$. Assuming henceforth that $x\in W^2_{x_\ast}$ and setting $\xi=x$ the Nagel, Stein, Wainger result shows that $\rho^{W^1_{x_\ast}}_\infty(x,y)\leq \delta$ for any $y\in B^{W^1_{x_\ast}}(x,\delta/C(W^2_{x_\ast}))\cap W^2_{x_\ast}$ and any $\delta>0$. In other words, $B^{W^1_{x_\ast}}(x,\delta/C(W^2_{x_\ast}))\cap W^2_{x_\ast}\subset B^{W^1_{x_\ast}}_\infty(x,\delta)$ so after writing $d_t=(\sqrt{t}/\sqrt{m})/C(W^2_{x_\ast})$ to condense notation this implies
\begin{align*}
\frac{\Lambda(W^1_{x_\ast},\{x_i\};x,\sqrt{t})}{|B^{W^1_{x_\ast}}_\infty(x,\sqrt{t}/\sqrt{m})|}
&=
\frac{|B^{W^1_{x_\ast}}(x,d_t)\cap W^2_{x_\ast}|}{|B^{W^1_{x_\ast}}_\infty(x,\sqrt{t}/\sqrt{m})|}
\frac{{|B^{W^1_{x_\ast}}(x,d_t)|}}{{|B^{W^1_{x_\ast}}(x,d_t)\cap W^2_{x_\ast}|}}\frac{\Lambda(W^1_{x_\ast},\{x_i\};x,\sqrt{t})}{|B^{W^1_{x_\ast}}(x,d_t)|}  \\
&\leq \frac{{|W^1_{x_\ast}|}}{{|B^{W^1_{x_\ast}}(x,\dist(x,\partial W^2_{x_\ast}))|}}\frac{\Lambda(W^1_{x_\ast},\{x_i\};x,\sqrt{t})}{|B^{W^1_{x_\ast}}(x,d_t)|}  
\end{align*}
since the first factor in the central expression is bounded by one and the second factor can be greater than one only when $B^{W^1_{x_\ast}}(x,d_t)$ intersects $\partial W^2_{x_\ast}$.\footnote{Here, $\dist$ refers to the distance induced by $\rho^{W^1_{x_\ast}}$.} Assuming henceforth that $x$ is restricted to any particular compact subset $K\subset W^2_{x_\ast}$, we can set $c_2(K)=\max_{K}|W^1_{x_\ast}|/|B^{W^1_{x_\ast}}(\cdot ,\dist(\cdot ,\partial W^2_{x_\ast}))|$ to estimate the first factor on the right uniformly. To estimate the second factor on the right, note that 
$\Lambda(W^1_{x_\ast},\{x_i\};x,\cdot)$ is a polynomial of bounded degree, so 
\[
\Lambda(W^1_{x_\ast},\{x_i\};x,\sqrt{t})=\Lambda(W^1_{x_\ast},\{x_i\};x,\sqrt{m}C(W^2_{x_\ast})d_t)\leq c_3(W^2_{x_\ast})\Lambda(W^1_{x_\ast},\{x_i\};x,d_t)
\] for some finite $c_3(W^2_{x_\ast})$ and therefore
\[
\frac{\Lambda(W^1_{x_\ast},\{x_i\};x,\sqrt{t})}{|B^{W^1_{x_\ast}}_\infty(x,\sqrt{t}/\sqrt{m})|}\leq 
c_2(K)c_3(W^2_{x_\ast})\frac{\Lambda(W^1_{x_\ast},\{x_i\};x,d_t)}{|B^{W^1_{x_\ast}}(x,d_t)|}.
\]
In the same article (\cite{MR793239}) Nagel, Stein and Wainger have also proved that for any compact $K'\subset W^1_{x_\ast}$, there exists a constant $c_4(K')$ such that 
\[
\Lambda(W^1_{x_\ast},\{x_i\};x,\delta)/|B^{W^1_{x_\ast}}(x,\delta)|\leq c_4(K')
\] for all $x\in K'$ and all $\delta>0$. Putting $\delta=d_t$, we have finally obtained the estimate 
\begin{equation}\label{asdfjkl;605}
\frac{\Lambda(W^1_{x_\ast},\{x_i\};x,\sqrt{t})}{|B^{W^1_{x_\ast}}_\infty(x,\sqrt{t}/\sqrt{m})|}\leq 
c_2(K)c_3(W^2_{x_\ast})c_4(K')
\end{equation}
for the first factor in line (\ref{asdfjkl;602}), provided that $x\in K\cap K'$ for any choice of compact sets $K\subset W^2_{x_\ast}$ and $K'\subset W^1_{x_\ast}$.

Proceeding to the second factor in line (\ref{asdfjkl;602}), we can write 
\[\frac{1}{\Lambda(W^1_{x_\ast},\{x_i\};x ,\sqrt{t})}=
\frac{\Lambda(W^1_{x_\ast},\{x_i\};x_\ast,\sqrt{t})}{\Lambda(W^1_{x_\ast},\{x_i\};x,\sqrt{t})}
\frac{1}{\Lambda(W^1_{x_\ast},\{x_i\};x_\ast,\sqrt{t})}\]
and estimate each factor on the right individually. In order to estimate 
\[
\Lambda(W^1_{x_\ast},\{x_i\};x_\ast,\sqrt{t})/\Lambda(W^1_{x_\ast},\{x_i\};x,\sqrt{t})
\]
consider the subcollection of multi-words $\frakI^{x_\ast}\subset \frakI$ such that $\det(\calI|_{x_\ast})\neq 0$ for $\calI\in \frakI^{x_\ast}$. Again, this is a nonempty subcollection according to the H\"{o}rmander condition. Within this subcollection, choose a multi-word $\calI^{x_\ast}=\{I^{x_\ast}_1,\ldots, I^{x_\ast}_n\}$ such that the associated exponent $\deg(\calI^{x_\ast})=|I^{x_\ast}_1|+\ldots+|I^{x_\ast}_n|$ is minimal (i.e. it is equal to $Q(x_\ast)$). Now with this choice of $\calI^{x_\ast}$ we have
\[
\frac{\Lambda(W^1_{x_\ast},\{x_i\};x_\ast,\sqrt{t})}{\Lambda(W^1_{x_\ast},\{x_i\};x,\sqrt{t})}=
\frac{\sum_{\calI\in\frakI}|\det(\calI|_{x_\ast})|t^{\deg(\calI)/2}}{\sum_{\calI\in\frakI}|\det(\calI|_x)|t^{\deg(\calI)/2}}
\leq 
\sum_{\calI\in\frakI^{x_\ast}}\frac{|\det(\calI|_{x_\ast})|t^{\deg(\calI)/2}}{|\det(\calI^{x_\ast}|_x)|t^{\deg(\calI^{x_\ast})/2}}
\]
on $\{x:\det(\calI^{x_\ast}|_{x})\neq 0\}$, which is an open neighborhood of $x_\ast$. Recall that $\calI^{x_\ast}$ was chosen so that the exponent $\deg(\calI^{\ast})$ is minimal among all $\calI\in\frakI^{x_\ast}$, and this means that \emph{every} exponent of $t$ which appears in the expression above is nonnegative, so that whenever $t^{\deg(\calI)/2}/t^{\deg(\calI^{x_\ast})/2}$ appears in this sum, it is bounded above by one provided that $t$ is. In other words, 
\[
\frac{\Lambda(W^1_{x_\ast},\{x_i\};x_\ast,\sqrt{t})}{\Lambda(W^1_{x_\ast},\{x_i\};x,\sqrt{t})}\leq 
\frac{\sum_{\calI\in\frakI^{x_\ast}}|\det(\calI|_{x_\ast})|}{|\det(\calI^{x_\ast}|_{x})|}
\]
on $\{x:\det(\calI^{x_\ast}|_{x})\neq 0\}$, provided that $0<t\leq 1$. Now since $\deg(\calI^{x_\ast})=Q(x_\ast)$,
\[
\Lambda(W^1_{x_\ast},\{x_i\};x_\ast,\sqrt{t})=\sum_{\calI\in \frakI^{x_\ast}}|\det(\calI|_{x_\ast})|t^{\deg(\calI)/2}\geq |\det(\calI^{x_\ast}|_{x_\ast})|t^{Q(x_\ast)/2}.
\]
From these two estimates follows the desired estimate of the second factor in line (\ref{asdfjkl;602}):
\begin{equation}\label{asdfjkl;606}
\frac{1}{\Lambda(W^1_{x_\ast},\{x_i\};x,\sqrt{t})}\leq 
\frac{\sum_{\calI\in\frakI^{x_\ast}}|\det(\calI|_{x_\ast})|}{|\det(\calI^{x_\ast}|_{x})|}\frac{1}{|\det(\calI^{x_\ast}|_{x_\ast})|t^{Q(x_\ast)/2}}
\end{equation}
in which the only $x$-dependent factor is $1/|\det(\calI^{x_\ast}|_{x})|$. Setting \\
$c_5(K'')=\max_{K''} |\det(\calI^{x_\ast}|_{(\cdot)})|^{-1}$ for any compact set $K''\subset \{x:\det(\calI^{x_\ast}|_{x})\neq 0\}$, we have after combining the estimates (\ref{asdfjkl;604}),(\ref{asdfjkl;605}) and (\ref{asdfjkl;606}) for lines (\ref{asdfjkl;601}) and (\ref{asdfjkl;602}), \begin{equation}
E^{L,\mu}_0(x;x,t)\leq\frac{ c_1c_2(K)c_3(W^2_{x_\ast})c_4(K')c_5(K'')\sum_{\calI\in\frakI^{x_\ast}}|\det(\calI|_{x_\ast})|}{(\inf_{W^1_{x_\ast}} h)|\det(\calI^{x_\ast}|_{x_\ast})|}t^{-Q(x_\ast)/2}
\end{equation}
for all $x\in K\cap K'\cap  K''$ and $0<t\leq 1$. This estimate holds, albeit with variable constants $c_2(K),c_3(W^2_{x_\ast}),c_4(K')$ and $c_5(K'')$, for any choice of compact sets $K\subset W^2_{x_\ast}$, $K'\subset W^1_{x_\ast}$ and $K''\subset \{x:\det(\calI^{x_\ast}|_{x})\neq 0\}$, so it is clear that $K,K'$ and $K''$ can be chosen so that $K\cap K'\cap K''$ contains a full open neighborhood $N_{x_\ast}$ of $x_\ast$. After such a choice has been made, setting
\[
k(x_\ast)=\frac{ c_1c_2(K)c_3(W^2_{x_\ast})c_4(K')c_5(K'')\sum_{\calI\in\frakI^{x_\ast}}|\det(\calI|_{x_\ast})|}{(\inf_{W^1_{x_\ast}} h)|\det(\calI^{x_\ast}|_{x_\ast})|}
\] completes the proof of Theorem \ref{asdfjkl;111}.

The underlying principle here is that the minimal number of brackets required to generate the tangent space at $x_\ast$ indicates the degree to which the operator $L$ will allow infinitessimal heat flow away from $x_\ast$: if more brackets are required to generate $T_{x_\ast}M$ than are required to generate $T_{x}M$ for $x$ near $x_\ast$ then it will be more difficult for heat to flow away from $x_\ast$ in small times and therefore it should be expected that $E^{L,\mu}_0(x_\ast;x_\ast,t)$ will have a more severe singularity as $t\rightarrow 0$ than $E^{L,\mu}_0(x;x,t)$, thus allowing $E^{L,\mu}_0(x;x,t)$ to be controlled in the $t\rightarrow 0$ limit by $E^{L,\mu}_0(x_\ast;x_\ast,t)$. On the other hand, as a simple consequence of the regularity of the vector fields $X_1,\ldots, X_k$, it is \emph{always} true that an equal or greater number of brackets are required to generate $T_{x_\ast}M$ than $T_xM$ for $x$ near $x_\ast$. This is the phenomenon which is borne out in the proof of the Theorem \ref{asdfjkl;111}, and in particular in the estimates which lead to (\ref{asdfjkl;606}). If $(X_{I^{x_\ast}_1}|_{x_\ast},\ldots, X_{I^{x_\ast}_n}|_{x_\ast})$ is a ``minimal frame" at $x_\ast$ in the sense that $\deg(\calI^{x_\ast})=Q(x_\ast)$, then it is \emph{not necessarily} a minimal frame at nearby points - but the redeeming feature is that, by continuity, the moving frame $x\mapsto (X_{I^{x_\ast}_1}|_{x},\ldots, X_{I^{x_\ast}_n}|_{x})$ must be nondegenerate in a full open neighborhood of $x_\ast$ and therefore it must  control the singularity of $E^{L,\mu}_0(x;x,t)$ as $t\rightarrow 0$ if $x$ is sufficiently near $x_\ast$. If a lower order frame exists at a point $x$ which is near $x_\ast$, it will only increase the propensity for heat to flow away from $x$, thereby making the diagonal singularity of $E^{L,\mu}_0(x;x,t)$ less severe and thus controllable by $E^{L,\mu}_0(x_\ast;x_\ast,t)$. This is what occurs if $x_\ast$ lies on the boundary of a level set of $Q$, for instance.

\section{Proofs of Theorems \ref{asdfjkl;114} and \ref{asdfjkl;107}}
For any $c\in\R$, $E_{c}^{L,\mu}=e^{-tc}E_0^{L,\mu}$. In particular, this applies for $c=\max V,\\\min V\in \R$ for any $V\in\scrC^\infty(M;\R)$. Elementary considerations involving extremal principals for parabolic operators demonstrate that
\[
e^{-t\max V}E_0^{L,\mu}=E_{\max V}^{L,\mu}\leq E_{V}^{L,\mu}\leq E_{\min V}^{L,\mu}=e^{-t\min V}E_0^{L,\mu}.
\]
Therefore,
\[
E_0^{L,\mu}-E^{L,\mu}_V\leq E^{L,\mu}_0-E_{\max V}^{L,\mu}=(1-e^{-t\max V})E^{L,\mu}_0=e^{-t\max V}(e^{t\max V}-1)E^{L,\mu}_0
\]
and by a similar computation, $E^{L,\mu}_V-E^{L,\mu}_0\leq (e^{-t\min V}-1)E^{L,\mu}_0$. Therefore, $|E^{L,\mu}_V-E^{L,\mu}_0|\leq e^{t\max |V|}(e^{t\max |V|}-1)E^{L,\mu}_0$. The first consequence of this estimate is the analog of Corollary \ref{asdfjkl;115} for nonzero $V\in\scrC^\infty(M;\R)$:
\begin{align}
t^{Q_L/2}E^{L,\mu}_V(x;x,t)&\leq t^{Q_L/2}|E^{L,\mu}_V(x;x,t)-E^{L,\mu}_0(x;x,t)|+t^{Q_L/2}E^{L,\mu}_0(x;x,t) \notag \\
&\leq e^{t\max |V|}(e^{t\max |V|}-1)t^{Q_L/2}E^{L,\mu}_0(x;x,t)+t^{Q_L/2}E^{L,\mu}_0(x;x,t) \notag \\
&\leq (e^{2\max |V|}+1)k_L\label{asdfjkl;119}
\end{align}
for $(x,t)\in M\times(0,1]$. Additionally, by Theorem \ref{asdfjkl;101},
\begin{align*}
|t^{Q_L/2}E^{L,\mu}_V(x;x,t)-\epsilon^{L,\mu}_0(x)|&\leq t^{Q_L/2}|E^{L,\mu}_V(x;x,t)-E^{L,\mu}_0(x;x,t)| \\
&\hspace{1cm}+|t^{Q_L/2}E^{L,\mu}_0(x;x,t)-\epsilon^{L,\mu}_0(x)| \\
&= t^{Q_L/2}|E^{L,\mu}_V(x;x,t)-E^{L,\mu}_0(x;x,t)|+|R_0^{L,\mu}(x,t)|\\
&\leq  e^{t\max |V|}(e^{t\max |V|}-1)t^{Q_L/2}E^{L,\mu}_0(x;x,t)+|R_0^{L,\mu}(x,t)| \\
&\leq  e^{t\max |V|}(e^{t\max |V|}-1)k_L+|R_0^{L,\mu}(x,t)| \\
&=o(1)
\end{align*}
pointwise as $t\rightarrow 0$. Thus, 
\begin{equation}\label{asdfjkl;120}
\lim_{t\rightarrow 0} t^{Q_L/2}E^{L,\mu}_V(x;x,t)=\epsilon^{L,\mu}_0(x)
\end{equation}
pointwise at every $x\in M$.

Using these estimates, we can prove Theorem \ref{asdfjkl;114} as follows:
\begin{align}
\lim_{t\rightarrow 0}t^{Q_L/2}\Tr_{L^2}e^{-t(L+V)}&= \lim_{t\rightarrow 0}\int t^{Q_L/2}E^{L,\mu}_V(x;x,t)\dmu(x)  \notag \\
&= \int \lim_{t\rightarrow 0}t^{Q_L/2}E^{L,\mu}_V(x;x,t)\dmu(x) \notag \\
&= \int \epsilon^{L,\mu}_0(x) \dmu(x) \label{asdfjkl;121}.
\end{align}
The first equality is a standard fact concerning kernels of trace class operators. The second equality is by dominated convergence, which is valid according to (\ref{asdfjkl;119}), and the third equality is a consequence of the pointwise limit (\ref{asdfjkl;120}). Now (\ref{asdfjkl;121}) is equivalent to the asymptotic equality 
\begin{equation}\label{asdfjkl;700}
\Tr_{L^2}e^{-t(L+V)}=\pars{\int \epsilon^{L,\mu}_0\dmu}t^{-Q_L/2}+o(t^{-Q_L/2})
\end{equation}
as $t\rightarrow 0$, and Theorem \ref{asdfjkl;114} is proved.

In Theorem \ref{asdfjkl;107}, $L$ is assumed to be self-adjoint and therefore $L+V$ is also self-adjoint since $V$ is real. If $P^+_\mu$ denotes the $L^2(M,\mu)$-orthonormal projection onto the eigenspaces of $L+V$ corresponding to strictly positive eigenvalues, then evidently
\begin{align*}
\lim_{\lambda\rightarrow \infty}N(\lambda,L+V)\lambda^{-Q_L/2}&=
\lim_{\lambda\rightarrow \infty}[N(\lambda,L+V)-N(0,L+V)]\lambda^{-Q_L/2}  \\
&=\frac{\lim_{t\rightarrow 0}t^{Q_L/2} \Tr_{L^2}(P^+_\mu e^{-t(L+V)})}{\Gamma(Q_L/2+1)}\\
&=\frac{\lim_{t\rightarrow 0}t^{Q_L/2} \Tr_{L^2} e^{-t(L+V)}}{\Gamma(Q_L/2+1)}\\
&=\frac{\int \epsilon^{L,\mu}_0\dmu}{\Gamma(Q_L/2+1)}.
\end{align*}
The first and third equalities hold because $L+V$ can have at most finitely many negative eigenvalues, the second equality is a routine application of Karamata's Tauberian theorem (see \cite{MR2273508}, e.g.), and the fourth equality follows from (\ref{asdfjkl;700}). This completes the proof of Theorem \ref{asdfjkl;107}.

\section{The Dirichlet Boundary Value Problem}
In this section we will consider the analogues of Theorems \ref{asdfjkl;114} and \ref{asdfjkl;107} for the Dirichlet boundary value problem in precompact domains in arbitrary smooth manifolds; i.e. $M$ will be a generic smooth manifold (not necessarily compact) and $L$ will be an operator on $M$ of the form
\begin{equation}\label{asdfjkl;131}
L=-\sum_{i=1}^{\infty}X_i^2 +\sum_{i,j=1}^\infty c_{ij}[X_i,X_j]+\sum_{i=1}^\infty \gamma_iX_i.
\end{equation}
It is assumed that the H\"{o}rmander condition is satisfied and that vector fields $\{X_i\}$ form a locally finite collection in the sense that every point in $M$ is contained in a neighborhood which intersects only finitely many of the supports of the $X_i$. Thus, there are no convergence issues involved with the infinite sums in (\ref{asdfjkl;131}), and furthermore any operator which is locally defined in the form (\ref{asdfjkl;131}) (with finite sums) can in fact be written in this form globally (with locally finite sums) within a perturbation by a smooth real potential. As before, $V$ denotes a generic element of $\scrC^\infty(M;\R)$ and $\mu$ denotes a generic smooth and nondegenerate volume density on $M$. 

For any connected open set $\Omega\subset \subset M$, the homogeneous dimension is defined just as it was in the compact case, except that we will only consider its values on $\ov{\Omega}$ (i.e. $Q_L^\Omega=\max_{\ov{\Omega}}Q(\cdot)$, $Q(\cdot)$ is locally bounded and integer-valued so the maximum is attained). The Dirichlet realization of $L+V$ on $L^2(\Omega,\mu)$ is constructed as follows. The scalar product on $\scrC^\infty_0(\Omega;\C)$ given by $(u,v)\mapsto \Re\angles{(L+V)u,v}_{L^2(\Omega,\mu)}+\kappa\angles{u,v}_{L^2(\Omega,\mu)}$ is sesquilinear for any $\kappa\in\R$ and positive-definite (hence Hermitian) provided that $\kappa$ is sufficiently large. In fact,  by the Poincar\'{e} inequality the resulting positive-definite scalar product is equivalent to $(u,v)\mapsto \sum_{i=1}^\infty\angles{X_iu,X_i v}_{L^2(\Omega, \mu)}$ (the sum is finite since $\ov{\Omega}$ is compact). Thus, we can define the anisotropic Sobolev space $H^1_0(\Omega,\mu)$ associated to $(L+V)|_{\scrC^\infty_0(\Omega;\C)}$ to be the completion of $\scrC^\infty_0(\Omega;\C)$ in the positive-definite Hermitian product $\angles{u,v}_{H^1_0(\Omega,\mu)}=\sum_{i=1}^\infty\angles{X_iu,X_i v}_{L^2(\Omega, \mu)}$.

Now for any $\kappa\in\R$, the sesquilinear scalar product \\$(u,v)\mapsto \angles{(L+V+\kappa)u,v}_{L^2(\Omega,\mu)}$ is continuous on $H^1_0(\Omega,\mu)$. Moreover, it is coercive provided that $\kappa$ is sufficiently large (as above), so the Lax-Milgram lemma ensures that $(L+V+\kappa)|_{\scrC^\infty_0(\Omega;\C)}$ extends to a continuous linear map of $H^1_0(\Omega,\mu)$ into its anti-dual $\ov{H^1_0(\Omega,\mu)}'$, which is a linear isomorphism for sufficiently large $\kappa$. However, $L^2(\Omega,\mu)$ embeds canonically into the anti-dual $\ov{H^1_0(\Omega,\mu)}'$ so the resulting closed operator $L+V+\kappa$ on $L^2(\Omega,\mu)$ with domain $\{u\in H^1_0(\Omega,\mu):(L+V+\kappa)u\in L^2(\Omega,\mu)\subset \ov{H^1_0(\Omega,\mu)}'\}$ is surjective for sufficiently large $\kappa$, and by R.S. Phillips' criterion (\cite{MR0104919}), this is enough to ensure that this extension of $-(L+V+\kappa)$ is maximally dissipative. 

It will also be of use to observe at this point that the adjoint of this realization of $L+V+\kappa$ has the same domain and is constructed from the formal adjoint of $L+V$ in the same manner. Thus, if $L$ is formally self-adjoint against the density $\mu$ (i.e. if it is symmetric on $\scrC^\infty_0(\Omega;\C)$) then this extension of $(L+V+\kappa)|_{\scrC^\infty_0(\Omega;\C)}$ is the Friedrichs' extension, in particular it is self-adjoint in $L^2(\Omega,\mu)$ for any $\kappa\in\R$.

Since $-(L+V+\kappa)$ with the given domain is maximally dissipative for sufficiently large $\kappa$, it  generates a contraction semigroup $e^{-t(L+V+\kappa)}$ on $L^2(\Omega,\mu)$ with smooth kernel $D_{V+\kappa}^{L,\mu}:\Omega\times \Omega\times (0,\infty)\rightarrow (0,\infty)$. The kernel for the semigroup $e^{-t(L+V)}$ corresponding to $\kappa=0$ (which may fail to be a contraction semigroup) is easily found to be $D_{V}^{L,\mu}(\cdot;\cdot,t)=e^{t\kappa}D_{V+\kappa}^{L,\mu}(\cdot;\cdot,t)$ and for $V\equiv 0$, Theorem \ref{asdfjkl;101} holds in $\ov{\Omega}$:
\begin{equation}\label{asdfjkl;147}
t^{Q^\Omega_L/2}D^{L,\mu}_0(x;x,t)=\sum_{j=0}^N \epsilon_{j/2}^{L,\mu}(x)t^{j/2}+S^{L,\mu}_{N/2}(x,t)
\end{equation}
and $S^{L,\mu}_{N/2}(x,t)=o(t^{N/2})$ pointwise. It is important to keep in mind that for any $x\in \Omega$, the coefficients $\epsilon_{j/2}^{L,\mu}(x)$ are the same as those described in the introduction for the global kernel $E^{L,\mu}_0$ on any compact manifold into which an open neighborhood of $x$ is embedded and on which $L$ and $\mu$ are smoothly extended,\footnote{This will be true provided that for a given extension, $Q$ is maximized on $\ov{\Omega}$ - if this is not the case then the coefficients will change by a translation in the index $j$. In any case this is a minor technicality.} they are determined only by the germ at $x$ of the coefficients of the operator and the volume density. 

To conclude this article we will prove for the Dirichlet realization of $L+V$ on $L^2(\Omega,\mu)$ the following analogues of Theorems \ref{asdfjkl;114} and \ref{asdfjkl;107}:

\begin{theorem} \label{asdfjkl;135} If $e^{-t(L+V)}$ denotes the semigroup generated by the Dirichlet realization of $L+V$ on $L^2(\Omega,\mu)$ then the asymptotic equality 
\begin{equation} 
\Tr_{L^2(\Omega,\mu)}e^{-t(L+V)}=\pars{\int_\Omega \epsilon^{L,\mu}_0\dmu}t^{-Q_L^\Omega/2}+o(t^{-Q_L^\Omega/2})
\end{equation}
holds as $t\rightarrow 0$.
\end{theorem}

\begin{theorem}\label{asdfjkl;136} If $L$ is formally self-adjoint in $\Omega$ and $N(\lambda,L+V)$ denotes the number of eigenvalues of the Friedrichs extension of $(L+V)|_{\scrC^\infty_0(\Omega;\C)}$ which are not greater than $\lambda$, then the asymptotic equality
\[
N(\lambda, L+V)= \frac{\int_\Omega \epsilon^{L,\mu}_0\dmu}{\Gamma(Q_L^\Omega/2+1)}\lambda^{Q_L^\Omega/2}+o(\lambda^{Q_L^\Omega/2})
\]
holds as $\lambda\rightarrow \infty$.
\end{theorem}

Again, a necessary and sufficient condition for the positivity of the nonnegative coefficient $\int_\Omega \epsilon^{L,\mu}_0\dmu$ is that $\{Q(\cdot)=Q_L^\Omega\}$ has positive measure in $\Omega$. Furthermore, we can assume without loss of generality that $M$ is compact to begin with, for if not then the problem can be reconstructed on a domain in a compact manifold as follows.  There exists a smoothly-bounded and connected domain $\Omega'$ such that $\Omega\subset \subset \Omega'\subset \subset M$, thus $\Omega\subset \Omega'_+\subset \ov{\Omega'_+}\uplus \ov{\Omega'_-}$ where $\ov{\Omega'_+}\uplus \ov{\Omega'_-}$ denotes the closed double of $\ov{\Omega'}$. The operator $L$ and volume density $\mu$ can then be smoothly extended from $\Omega\subset \Omega'_+$ to the entirety of $\ov{\Omega'_+}\uplus \ov{\Omega'_-}$. Furthermore, the set $\{Q(\cdot)\leq Q^\Omega_L\}$ is open in $\ov{\Omega'_+}\uplus \ov{\Omega'_-}$ and it contains the compact set $\ov{\Omega}$, and this means that the extended operator $L$ can be further modified if necessary in the complement of $\ov{\Omega}$ so that $Q$ is maximized on $\ov{\Omega}$. This will ensure that if $E^{L,\mu}_0$ denotes the heat kernel on the compact manifold $\ov{\Omega'_+}\uplus \ov{\Omega'_-}$ associated to these extensions of $L$ and $\mu$, then the coefficients $\epsilon_{j/2}^{L,\mu}$ appearing in (\ref{asdfjkl;124}) (which holds for $E^{L,\mu}_0$ since $\ov{\Omega'_+}\uplus \ov{\Omega'_-}$ is compact) and (\ref{asdfjkl;147}) coincide on $\ov{\Omega}$.

Now Theorem \ref{asdfjkl;136} will follow as a direct corollary of Theorem \ref{asdfjkl;135}, which in turn will follow for nonzero $V\in\scrC^\infty(M;\R)$ if it is proved for the Dirichlet realization of $L$ alone, with no potential. These facts are proved for $\Omega$ as in section 3 where they are proved for compact $M$. In particular the parabolic extremal principle plays the same role here as it does in section 3, except that in this case the parabolic boundary will contain nontrivial lateral components since $\partial \Omega$ is nonempty. However, the inequality
\[
e^{-t\max_{\ov{\Omega}} V}D_0^{L,\mu}=D_{\max_{\ov{\Omega}} V}^{L,\mu}\leq D_{V}^{L,\mu}\leq D_{\min_{\ov{\Omega}} V}^{L,\mu}=e^{-t\min_{\ov{\Omega}} V}D_0^{L,\mu}
\]  
will still hold because the Dirichlet boundary condition forces all of these heat kernels to vanish on the lateral boundary. The problem is thus reduced to bounding $t^{Q_L^\Omega/2}D^{L,\mu}_0(x;x,t)$ by a constant, uniformly for $(x,t)\in \Omega\times (0,1]$.

Since the domain $\Omega$ is an open manifold, there is no straightforward way to get such an estimate as there is in the compact case. However, the desired estimate can be deduced by comparing $D^{L,\mu}_0$ with the global heat kernel $E^{L,\mu}_0$. Since the global kernel is positive, the difference $E^{L,\mu}_0-D^{L,\mu}_0$ is nonnegative on the entire parabolic boundary. This means that the parabolic extremal principle can be applied yet again to demonstrate that $E^{L,\mu}_0\geq D^{L,\mu}_0$ in the entirety of $\Omega\times (0,\infty)$ whence $t^{Q_L^\Omega/2}D^{L,\mu}_0(x;x,t)\leq t^{Q_L^\Omega/2}E^{L,\mu}_0(x;x,t)\leq k_L^\Omega$ in $\Omega\times (0,1]$ by Corollary \ref{asdfjkl;115}, which holds for the global kernel $E^{L,\mu}_0$. Theorems \ref{asdfjkl;135} and \ref{asdfjkl;136} follow from this by proofs which are analogous to those presented in section 3 for Theorems \ref{asdfjkl;114} and \ref{asdfjkl;107}, respectively.

\bibliographystyle{amsalpha}

\bibliography{alubib}

\def\cprime{$'$}
\providecommand{\bysame}{\leavevmode\hbox to3em{\hrulefill}\thinspace}
\providecommand{\MR}{\relax\ifhmode\unskip\space\fi MR }
\providecommand{\MRhref}[2]{%
  \href{http://www.ams.org/mathscinet-getitem?mr=#1}{#2}
}
\providecommand{\href}[2]{#2}
\begin{thebibliography}{BAL91b}

\bibitem[BA88]{MR962859}
G{\'e}rard Ben~Arous, \emph{Noyau de la chaleur hypoelliptique et g\'eom\'etrie
  sous-riemannienne}, Stochastic analysis ({P}aris, 1987), Lecture Notes in
  Math., vol. 1322, Springer, Berlin, 1988, pp.~1--16. \MR{962859 (89k:35054)}

\bibitem[BA89]{MR1011978}
\bysame, \emph{D\'eveloppement asymptotique du noyau de la chaleur
  hypoelliptique sur la diagonale}, Ann. Inst. Fourier (Grenoble) \textbf{39}
  (1989), no.~1, 73--99. \MR{1011978 (91b:58272)}

\bibitem[BAL91a]{MR1128069}
G.~Ben~Arous and R.~L{\'e}andre, \emph{D\'ecroissance exponentielle du noyau de
  la chaleur sur la diagonale. {I}}, Probab. Theory Related Fields \textbf{90}
  (1991), no.~2, 175--202. \MR{1128069 (93b:60136a)}

\bibitem[BAL91b]{MR1133372}
\bysame, \emph{D\'ecroissance exponentielle du noyau de la chaleur sur la
  diagonale. {II}}, Probab. Theory Related Fields \textbf{90} (1991), no.~3,
  377--402. \MR{1133372 (93b:60136b)}

\bibitem[BGV04]{MR2273508}
Nicole Berline, Ezra Getzler, and Mich{\`e}le Vergne, \emph{Heat kernels and
  {D}irac operators}, Grundlehren Text Editions, Springer-Verlag, Berlin, 2004,
  Corrected reprint of the 1992 original. \MR{2273508 (2007m:58033)}

\bibitem[FP80]{MR589278}
C.~Fefferman and D.~H. Phong, \emph{On the asymptotic eigenvalue distribution
  of a pseudodifferential operator}, Proc. Nat. Acad. Sci. U.S.A. \textbf{77}
  (1980), no.~10, part 1, 5622--5625. \MR{589278 (82h:35101)}

\bibitem[FP83]{MR730094}
\bysame, \emph{Subelliptic eigenvalue problems}, Conference on harmonic
  analysis in honor of {A}ntoni {Z}ygmund, {V}ol. {I}, {II} ({C}hicago, {I}ll.,
  1981), Wadsworth Math. Ser., Wadsworth, Belmont, CA, 1983, pp.~590--606.
  \MR{730094 (86c:35112)}

\bibitem[Gil04]{MR2040963}
Peter~B. Gilkey, \emph{Asymptotic formulae in spectral geometry}, Studies in
  Advanced Mathematics, Chapman \& Hall/CRC, Boca Raton, FL, 2004. \MR{2040963
  (2005f:58041)}

\bibitem[JSC86]{MR865430}
David~S. Jerison and Antonio S{\'a}nchez-Calle, \emph{Estimates for the heat
  kernel for a sum of squares of vector fields}, Indiana Univ. Math. J.
  \textbf{35} (1986), no.~4, 835--854. \MR{865430 (88c:58064)}

\bibitem[Men79]{MR547016}
A.~Menikoff, \emph{Hypoelliptic operators with double characteristics}, Seminar
  on {S}ingularities of {S}olutions of {L}inear {P}artial {D}ifferential
  {E}quations ({I}nst. {A}dv. {S}tudy, {P}rinceton, {N}.{J}., 1977/78), Ann. of
  Math. Stud., vol.~91, Princeton Univ. Press, Princeton, N.J., 1979,
  pp.~65--79. \MR{547016 (81i:35036)}

\bibitem[M{\'e}t76]{MR0427858}
Guy M{\'e}tivier, \emph{Fonction spectrale et valeurs propres d'une classe
  d'op\'erateurs non elliptiques}, Comm. Partial Differential Equations
  \textbf{1} (1976), no.~5, 467--519. \MR{0427858 (55 \#888)}

\bibitem[Mon02]{MR1867362}
Richard Montgomery, \emph{A tour of subriemannian geometries, their geodesics
  and applications}, Mathematical Surveys and Monographs, vol.~91, American
  Mathematical Society, Providence, RI, 2002. \MR{1867362 (2002m:53045)}

\bibitem[MS67]{MR0217739}
H.~P. McKean, Jr. and I.~M. Singer, \emph{Curvature and the eigenvalues of the
  {L}aplacian}, J. Differential Geometry \textbf{1} (1967), no.~1, 43--69.
  \MR{0217739 (36 \#828)}

\bibitem[MS78a]{MR520875}
A.~Menikoff and J.~Sj{\"o}strand, \emph{The eigenvalues of hypoelliptic
  operators}, \'{E}quations aux d\'eriv\'ees partielles ({P}roc. {C}onf.,
  {S}aint-{J}ean-de-{M}onts, 1977), Lecture Notes in Math., vol. 660, Springer,
  Berlin, 1978, pp.~157--163. \MR{520875 (82b:35120)}

\bibitem[MS78b]{MR0481627}
\bysame, \emph{On the eigenvalues of a class of hypoelliptic operators}, Math.
  Ann. \textbf{235} (1978), no.~1, 55--85. \MR{0481627 (58 \#1735)}

\bibitem[MS79a]{MR555302}
\bysame, \emph{The eigenvalues of hypoelliptic operators. {III}. {T}he
  nonsemibounded case}, J. Analyse Math. \textbf{35} (1979), 123--150.
  \MR{555302 (82m:35115)}

\bibitem[MS79b]{MR564905}
\bysame, \emph{On the eigenvalues of a class of hypoelliptic operators. {II}},
  Global analysis ({P}roc. {B}iennial {S}em. {C}anad. {M}ath. {C}ongr., {U}niv.
  {C}algary, {C}algary, {A}lta., 1978), Lecture Notes in Math., vol. 755,
  Springer, Berlin, 1979, pp.~201--247. \MR{564905 (82m:35114)}

\bibitem[NSW85]{MR793239}
Alexander Nagel, Elias~M. Stein, and Stephen Wainger, \emph{Balls and metrics
  defined by vector fields. {I}. {B}asic properties}, Acta Math. \textbf{155}
  (1985), no.~1-2, 103--147. \MR{793239 (86k:46049)}

\bibitem[Phi59]{MR0104919}
R.~S. Phillips, \emph{Dissipative operators and hyperbolic systems of partial
  differential equations}, Trans. Amer. Math. Soc. \textbf{90} (1959),
  193--254. \MR{0104919 (21 \#3669)}

\bibitem[RS76]{MR0436223}
Linda~Preiss Rothschild and E.~M. Stein, \emph{Hypoelliptic differential
  operators and nilpotent groups}, Acta Math. \textbf{137} (1976), no.~3-4,
  247--320. \MR{0436223 (55 \#9171)}

\bibitem[Tak88]{MR944857}
Satoshi Takanobu, \emph{Diagonal short time asymptotics of heat kernels for
  certain degenerate second order differential operators of {H}\"ormander
  type}, Publ. Res. Inst. Math. Sci. \textbf{24} (1988), no.~2, 169--203.
  \MR{944857 (89i:35026)}

\end{thebibliography}

\end{document}